\documentclass[10pt]{amsart}
\usepackage{amssymb}
\usepackage{yhmath}

\usepackage{CJK}
\usepackage{etex}
\usepackage{subfigure} 
\usepackage{xmpmulti}  
\usepackage{colortbl,dcolumn}     
\usepackage[all,cmtip,line]{xy}

\usepackage{tikz}

\usetikzlibrary{calc}
\usetikzlibrary{through}

\newtheorem{proposition}{Proposition}[section]
\newtheorem{lemma}[proposition]{Lemma}
\newtheorem{corollary}[proposition]{Corollary}
\newtheorem{theorem}[proposition]{Theorem}

\theoremstyle{definition}

\newcommand{\selabel}[1]{\label{se:#1}}
\newcommand{\seref}[1]{Section~\ref{se:#1}}

\def\<{\leqslant}
\def\>{\geqslant}
\def\a{\alpha}
\def\b{\beta}
\def\d{\delta}

\def\ol{\overline}
\def\t{\triangle}
\def\e{\varepsilon}

\def\s{\sigma}

\def\ot{\otimes}
\def\ra{\rightarrow}

\date{}
\begin{document}
\title[The Projective Class Rings of a family of pointed Hopf algebras]
{The Projective Class Rings of a family of pointed Hopf algebras of Rank two}
\author{Hui-Xiang Chen}
\address{School of Mathematical Science, Yangzhou University,
Yangzhou 225002, China}
\email{hxchen@yzu.edu.cn}
\author{Hassan Suleman Esmael Mohammed}
\address{School of Mathematical Science, Yangzhou University,
Yangzhou 225002, China}
\email{esmailhassan313@yahoo.com}
\author{Weijun Lin}
\address{School of Mathematical Science, Yangzhou University,
Yangzhou 225002, China}
\email{wjlin@yzu.edu.cn}
\author{Hua Sun}
\address{School of Mathematical Science, Yangzhou University,
Yangzhou 225002, China}
\email{997749901@qq.com}
\begin{abstract}
In this paper, we compute the projective class rings of the tensor product
$\mathcal{H}_n(q)=A_n(q)\ot A_n(q^{-1})$ of Taft algebras $A_n(q)$ and $A_n(q^{-1})$, and its
cocycle deformations $H_n(0,q)$ and $H_n(1,q)$, where $n>2$ is a positive integer
and $q$ is a primitive $n$-th root of unity. It is shown that the projective class rings $r_p(\mathcal{H}_n(q))$,
$r_p(H_n(0,q))$ and $r_p(H_n(1,q))$ are commutative rings generated by three elements, three elements and
two elements subject to some relations, respectively.
It turns out that even $\mathcal{H}_n(q)$, $H_n(0,q)$ and $H_n(1,q)$
are cocycle twist-equivalent to each other, they are of different representation types: wild, wild and tame,
respectively.
\end{abstract}
\maketitle

\section{\bf Introduction}\selabel{1}

Let $H$ be a Hopf algebra over a field $\mathbb K$. Doi \cite{Doi} introduced a cocycle twisted Hopf
algebra $H^{\s}$ for a convolution invertible 2-cocycle $\s$ on $H$. It is shown in \cite{DoiTak, Maj92} that
the Drinfeld double $D(H)$ is a cocycle twisting of the tensor product Hopf algebra
$H^{*cop}\ot H$. The 2-cocycle twisting is extensively employed in various
researches. For instance, Andruskiewitsch et al. \cite{AndrFanGarVen} considered the twists of Nichols algebras
associated to racks and cocycles. Guillot, Kassel and Masuoka \cite{GuiKasMas} obtained some examples
by twisting comodule algebras by 2-cocycles.
It is well known that the monoidal category $\mathcal{M}^H$ of right
$H$-comodules is equivalent to the monoidal category $\mathcal{M}^{H^{\s}}$ of right
$H^{\s}$-comodules. On the other hand, we know that the braided monoidal category
$_H{\mathcal YD}^H$ of Yetter-Drinfeld $H$-modules is the center of the monoidal
category $\mathcal{M}^H$ for any Hopf algebra $H$ (e.g., see \cite{Ka}).
Hence the monoidal equivalence from $\mathcal{M}^H$ to $\mathcal{M}^{H^{\s}}$
gives rise to a braided monoidal equivalence
from $_H{\mathcal YD}^H$ to $_{H^{\s}}{\mathcal YD}^{H^{\s}}$.
Chen and Zhang \cite{ChenZhang} described a braided monoidal equivalent functor from
$_H{\mathcal YD}^H$ to $_{H^{\s}}{\mathcal YD}^{H^{\s}}$.
Benkart et al. \cite{BenkPerWith} used a result of Majid and Oeckl \cite{MajOec}
to give a category equivalence between Yetter-Drinfeld modules for a finite-dimensional
pointed Hopf algebra $H$ and those for its cocycle twisting $H^{\s}$.
However, the Yetter-Drinfeld module category $_H{\mathcal YD}^H$
is also the center of the monoidal category $_H\mathcal M$ of left $H$-modules.
This gives rise to a natural question:

Is there any relations between the two monoidal
categories $_H\mathcal M$ and $_{H^{\s}}\mathcal M$ of left modules over two cocycle twist-equivalent Hopf algebras
$H$ and $H^{\s}$? or how to detect the two monoidal categories $_H\mathcal M$ and $_{H^{\s}}\mathcal M$?

This article seeks to address this question through investigating the representation types and
projective class rings of a family of pointed Hopf algebras of rank 2, the tensor products of two Taft
algebras, and their two cocycle deformations.

In the investigation of the monoidal category of modules over a Hopf algebra $H$,
the decomposition problem of tensor products of indecomposables is of most importance
and has received enormous attentions. Our approach is to explore the representation type
of $H$ and the projective class ring of $H$, which is a subring of
the representation ring (or Green ring) of $H$. Originally, the concept of the Green ring $r(H)$
stems from the modular representations of finite groups (see \cite{Green}, etc.)
Since then, there have been plenty of works on the Green rings.
For finite-dimensional group algebras, one can refer to \cite{Archer, BenCar, BenPar, BrJoh, HTW}.
For Hopf algebras and quantum groups, one can see \cite{ChVOZh, Chin, Cib, LiZhang, Wakui, With}.

The $n^4$-dimensional Hopf algebra $H_n(p, q)$ was introduced in \cite{Ch1},
where $n\>2$ is an integer, $q\in\mathbb K$ is a primitive $n$-th root of unity and $p\in\mathbb K$.
If $p\neq 0$, then $H_n(p, q)$ is isomorphic to the Drinfeld double $D(A_n(q^{-1}))$ of the Taft
algebra $A_n(q^{-1})$. In particular, we have $H_n(p, q)\cong H_n(1, q)\cong D(A_n(q^{-1}))$ for any $p\neq 0$.
Moreover, $H_n(p,q)$ is a cocycle deformation of $A_n(q)\ot A_n(q^{-1})$.
For the details, the reader is directed to \cite{Ch1, Ch2}.
When $n=2$ ($q=-1$), $A_2(-1)$ is exactly the Sweedler 4-dimensional Hopf algebra $H_4$.
Chen studied the finite dimensional representations of $H_n(1,q)$ in \cite{Ch2, Ch4},
and the Green ring $r(D(H_4))$ in \cite{Ch5}.
Using a different method, Li and Hu \cite{LiHu} also studied the finite dimensional representations of
the Drinfeld double $D(H_4)$, the Green ring $r(D(H_4))$ and the projective class ring $p(D(H_4))$.
They also studied two Hopf algebras which are cocycle deformations of $D(H_4)$.
By \cite{Ch4}, one knows that $D(H_4)$ is of tame representation type.
By \cite{LiHu}, the two cocycle deformations of $D(H_4)$ are also of tame representation type.

In this paper, we study the three cocycle twist-equivalent Hopf algebras
$\mathcal{H}_n(q)=A_n(q)\ot A_n(q^{-1})$, $H_n(0,q)$ and $H_n(1,q)$ by investigating their
representation types and projective class rings, where $n\>3$.
In \seref{2}, we introduce the Taft algebras $A_n(q)$, the tensor product
$\mathcal{H}_n(q)=A_n(q)\ot A_n(q^{-1})$ and the Hopf algebras $H_n(p,q)$.
In \seref{3}, we first show that $\mathcal{H}_n(q)$ is of wild representation type.
With a complete set of orthogonal primitive idempotents, we classify the simple modules and indecomposable
projective modules over $\mathcal{H}_n(q)$, and decompose the tensor products of these modules.
This leads the description of the projective class ring $r_p(\mathcal{H}_n(q))$,
the Jacobson radical $J(R_p(\mathcal{H}_n(q)))$ of the projective class algebra $R_p(\mathcal{H}_n(q))$
and the quotient algebra $R_p(\mathcal{H}_n(q))/J(R_p(\mathcal{H}_n(q)))$.
In \seref{4}, we first show that $H_n(0,q)$ is a symmetric algebra of wild representation type.
Then we give a complete set of orthogonal primitive idempotents with the Gabriel quiver,
and classify the simple modules and indecomposable projective modules over $H_n(0,q)$.
We also describe the projective class ring $r_p(H_n(0,q))$, the Jacobson radical $J(R_p(H_n(0,q)))$
of the projective class algebra $R_p(H_n(0,q))$ and the quotient algebra  $R_p(H_n(0,q))/J(R_p(H_n(0,q)))$.
In \seref{5}, using the decompositions of tensor products of indecomposables over $H_n(1,q)$ given
in \cite{ChenHassenSun}, we describe the structure of the projective class ring $r_p(H_n(1,q))$.
It is interesting to notice that even the Hopf algebras $\mathcal{H}_n(q)$, $H_n(0,q)$ and $H_n(1,q)$
are cocycle twist-equivalent to each other, they own the different number of blocks
with 1, $n$ and $\frac{n(n+1)}{2}$, respectively  (see \cite[Corollary 2.7]{Ch4} for $H_n(1,q)$).
$\mathcal{H}_n(q)$ and $H_n(0,q)$ are basic algebras of wild representation type, but $H_n(1,q)$ is not
basic and is of tame representation type. $H_n(0,q)$ and $H_n(1,q)$ are symmetric algebras, but $\mathcal{H}_n(q)$
is not.

\section{\bf Preliminaries}\selabel{2}

Throughout, we work over an algebraically closed field $\mathbb K$.
Unless otherwise stated, all algebras, Hopf algebras and modules are
defined over $\mathbb K$; all modules are left modules and finite dimensional;
all maps are $\mathbb K$-linear; dim and $\otimes$ stand for ${\rm dim}_{\mathbb K}$ and
$\ot_{\mathbb K}$, respectively. Given an algebra $A$, $A$-mod denotes the category of
finite-dimensional $A$-modules. For any $A$-module $M$ and nonnegative integer $l$,
let $lM$ denote the direct sum of $l$ copies of $M$. For the theory of Hopf algebras and quantum groups, we
refer to \cite{Ka, Maj, Mon, Sw}.
Let $\mathbb Z$ denote all integers, and ${\mathbb Z}_n={\mathbb Z}/n{\mathbb Z}$.

Let $H$ be a Hopf algebra. The Green ring $r(H)$ of $H$
can be defined as follows. $r(H)$ is the abelian group generated by the
isomorphism classes $[M]$ of $M$ in $H$-mod
modulo the relations $[M\oplus V]=[M]+[V]$. The multiplication of $r(H)$
is given by the tensor product of $H$-modules, that is,
$[M][V]=[M\ot V]$. Then $r(H)$ is an associative ring.
The projective class ring $r_p(H)$ of $H$ is the subring of $r(H)$ generated by
projective modules and simple modules (see \cite{Cib99}).
Then the Green algebra $R(H)$ and projective algebra $R_p(H)$ are associative $\mathbb K$-algebras
defined by $R(H):=\mathbb{K}\ot_{\mathbb Z}r(H)$ and $R_p(H):=\mathbb{K}\ot_{\mathbb Z}r_p(H)$, respectively.
Note that $r(H)$ is a free abelian group with a $\mathbb Z$-basis
$\{[V]|V\in{\rm ind}(H)\}$, where ${\rm ind}(H)$ denotes the category
of finite dimensional indecomposable $H$-modules.

The Grothendieck ring $G_0(H)$ of $H$ is defined similarly.
$G_0(H)$ is the abelian group generated by the isomorphism classes $[M]$ of $M$ in $H$-mod
modulo the relations $[M]=[N]+[V]$ for any short exact sequence $0\ra N\ra M\ra V\ra 0$ in $H$-mod.
The multiplication of $G_0(H)$ is given by the tensor product of $H$-modules, that is,
$[M][V]=[M\ot V]$. Then $G_0(H)$ is also an associative ring.
Moreover, there is a canonical ring epimorphism from $r(H)$ onto $G_0(H)$.

Let $n\>2$ be an integer and $q\in\mathbb K$ a primitive $n$-th root of unity.
Then the $n^2$-dimensional Taft Hopf algebra $A_n(q)$ is defined as follows
(see \cite{Ta}): as an algebra, $A_n(q)$ is generated by $g$ and $x$ with relations
$$g^n=1,\ x^n=0, \ xg=qgx.$$
The coalgebra structure and antipode are given by
$$\begin{array}{c}
\triangle (g)=g\otimes g,\ \triangle (x)=x\otimes g+1\otimes x,\
\varepsilon (g)=1,\ \varepsilon (x)=0,\\
S(g)=g^{-1}=g^{n-1},\ S(x)=-xg^{-1}=-q^{-1}g^{n-1}x.\\
\end{array}$$
Since $q^{-1}$ is also a primitive $n$-th root of unity,
one can define another Taft Hopf algebra $A_n(q^{-1})$, which is generated, as an algebra,
by $g_1$ and $x_1$ with relations $g_1^n=1$, $x_1^n=0$ and $x_1g_1=q^{-1}g_1x_1$. The coalgebra structure
and antipode are given similarly to $A_n(q)$.
Then $A_n(q^{-1})\cong A_n(q)^{\rm op}$ as Hopf algebras.

The first author Chen introduced a Hopf algebra $H_n(p,q)$ in \cite{Ch1}, where $p, q\in\mathbb K$ and
$q$ is a primitive $n$-th root of unity. It was shown there that $H_n(p,q)$ is isomorphic to
a cocycle deformation of the tensor product $A_n(q)\ot A_n(q^{-1})$.

The tensor product $A_n(q)\ot A_n(q^{-1})$ can be described as follows.
Let $\mathcal{H}_n(q)$ be the algebra generated by $a, b, c$ and $d$ subject to the relations:
$$\begin{array}{lllll}
ba=qab,& db=bd, & ca=ac,& dc=qcd,& cb=bc,\\
a^n=0, & b^n=1, &c^n=1,& d^n=0, & da=ad.
\end{array}$$
Then $\mathcal{H}_n(q)$ is a Hopf algebra with the coalgebra structure and antipode given by
$$\begin{array}{lll}
\t(a)=a\otimes b+1\otimes a, & \e(a)=0, & S(a)=-ab^{-1}=-ab^{n-1},\\
\t(b)=b\otimes b, & \e(b)=1, & S(b)=b^{-1}=b^{n-1},\\
\t(c)=c\otimes c,& \e(c)=1, & S(c)=c^{-1}=c^{n-1},\\
\t(d)=d\otimes c+1\otimes d,& \e(d)=0, & S(d)=-dc^{-1}=-dc^{n-1}.
\end{array}$$
It is straightforward to verify that there is a Hopf algebra isomorphism from $\mathcal{H}_n(q)$ to
$A_n(q)\ot A_n(q^{-1})$ via $a\mapsto 1\ot x_1$, $b\mapsto 1\ot g_1$, $c\mapsto g\ot 1$ and $d\mapsto x\ot 1$.
Obviously, $\mathcal{H}_n(q)$ is $n^4$-dimensional with a $\mathbb K$-basis $\{a^ib^jc^ld^k|0\<i, j, l, k\<n-1\}$.

Let $p\in\mathbb K$. Then one can define another $n^4$-dimensional Hopf algebra $H_n(p, q)$, which is
generated as an algebra by $a, b, c$ and $d$ subject to the relations:
$$\begin{array}{lllll}
ba=qab,& db=qbd, & ca=qac,& dc=qcd,& bc=cb,\\
a^n=0, & b^n=1, &c^n=1,& d^n=0, & da-qad=p(1-bc).
\end{array}$$
The coalgebra structure and antipode are defined in the same way as $\mathcal{H}_n(q)$ before.
$H_n(p, q)$ has a $\mathbb K$-basis $\{a^ib^jc^ld^k|0\<i, j, l, k\<n-1\}$.
When $p\neq 0$, $H_n(p, q)\cong H_n(1, q)\cong D(A_n(q^{-1}))$ (see \cite{Ch1, Ch2}).
If $n=2$ ($q=-1$), then $H_2(1, -1)\cong D(H_4)$, and
$H_2(0,-1)$ is exactly the Hopf algebra $\ol{\mathcal A}$ in \cite{LiHu}.

By \cite[Lemma 3.2]{Ch1}, there is an invertible skew-pairing
$\tau_p: A_n(q)\ot A_n(q^{-1})\ra \mathbb K$ given by
$\tau_p(g^ix^j, x_1^kg_1^l)=\d_{jk}p^jq^{il}(j)!_q$, $0\<i,j,k,l<n$.
Hence one can form a double crossproduct $A_n(q)\bowtie_{\tau_p}A_n(q^{-1})$.
Moreover, $A_n(q)\bowtie_{\tau_p}A_n(q^{-1})$ is isomorphic to $H_n(p,q)$
as a Hopf algebra (see \cite[Theorem 3.3]{Ch1}). By \cite{DoiTak}, $\tau_p$ induces
an invertible 2-cocycle $[\tau_p]$ on $A_n(q)\ot A_n(q^{-1})$ such that
$A_n(q)\bowtie_{\tau_p}A_n(q^{-1})=(A_n(q)\ot A_n(q^{-1}))^{[\tau_p]}$.
Thus, there is a corresponding invertible 2-cocycle $\s_p$ on $\mathcal{H}_n(q)$
such that $\mathcal{H}_n(q)^{\s_p}\cong H_n(p,q)$ as Hopf algebras.
In particular, we have $\mathcal{H}_n(q)^{\s_0}\cong H_n(0,q)$ and $\mathcal{H}_n(q)^{\s_1}\cong H_n(1,q)$.
In general, if $\s$ is a convolution invertible 2-cocycle on a Hopf algebra $H$,
then $\s^{-1}$ is an invertible 2-cocycle on $H^{\s}$
and $(H^{\s})^{\s^{-1}}=H$ (see \cite[Lemma 1.2]{chen99}).
More generally, if $\s$ is an invertible 2-cocycle on $H$ and $\tau$ is an invertible 2-cocycle on $H^{\s}$,
then $\tau*\s$ is an invertible 2-cocycle on $H$ and $H^{\tau*\s}=(H^{\s})^{\tau}$ (see \cite[Lemma 1.4]{chen99}).
Thus, the Hopf algebras $\mathcal{H}_n(q)$, $H_n(0,q)$ and $H_n(1,q)$ are
cocycle twist-equivalent to each other.

Throughout the following, fix an integer $n>2$ and let $q\in\mathbb K$ be
a primitive $n$-th root of unity. For any $m\in\mathbb Z$, denote still by $m$ the image of $m$ under the canonical
projection $\mathbb{Z}\ra\mathbb{Z}_n=\mathbb{Z}/n\mathbb{Z}$.

\section{\bf The Projective Class Ring of $\mathcal{H}_n(q)$}\selabel{3}

In this section, we investigate the representations and the projective class ring of
$\mathcal{H}_n(q)$, or equivalently, of $A_n(q)\ot A_n(q^{-1})$.

Let $A$ be the subalgebra of $\mathcal {H}_n(q)$ generated by $a$ and $d$.
Then $A$ is isomorphic to the quotient algebra $\mathbb{K}[x, y]/(x^n, y^n)$
of the polynomial algebra $\mathbb{K}[x, y]$ modulo the ideal
$(x^n,y^n)$ generated by $x^n$ and $y^n$.
Let $G=G(\mathcal{H}_n(q))$ be the group of group-like elements of $\mathcal{H}_n(q)$.
Then $G=\{b^ic^j|i,j\in\mathbb{Z}_n\}\cong \mathbb{Z}_n\times\mathbb{Z}_n$,
and $\mathbb{K}G=\mathcal{H}_n(q)_0$, the coradical of $\mathcal{H}_n(q)$.
Clearly, $A$ is a left $\mathbb{K}G$-module algebra with the action given by
$b\cdot a=qa$, $b\cdot d=d$, $c\cdot a=a$ and $c\cdot d=q^{-1}d$.
Hence one can form a smash product algebra $A\#\mathbb{K}G$. It is easy to see that
$\mathcal{H}_n(q)$ is isomorphic to $A\#\mathbb{K}G$ as an algebra.
Since $n\>3$, it follows from \cite[p.295(3.4)]{Ringel} that $A$
is of wild representation type. Since char$(\mathbb{K})\nmid |G|$,
$\mathbb{K}G$ is a semisimple and cosemisimple Hopf algebra. It follows from
\cite[Theorem 4.5]{Liu} that $A\#\mathbb{K}G$ is of wild representation type.
As a consequence, we obtain the following result.

\begin{proposition}\label{3.1}
$\mathcal{H}_n(q)$ is of wild representation type.
\end{proposition}

$\mathcal{H}_n(q)$ has $n^2$ orthogonal primitive idempotents
$$\begin{array}{c}
e_{i,j}=\frac{1}{n^2}\sum_{k,l\in\mathbb{Z}_n}q^{-ik-jl}b^kc^l
=\frac{1}{n^2}\sum_{k,l=0}^{n-1}q^{-ik-jl}b^kc^l,\ \ i,j\in\mathbb{Z}_n.\\
\end{array}$$

\begin{lemma}\label{3.2}
Let $i,j\in\mathbb{Z}_n$. Then
$$be_{i,j}=q^ie_{i,j}, \ ce_{i,j}=q^je_{i,j},\
ae_{i,j}=e_{i+1,j}a, \ de_{i,j}=e_{i,j-1}d.$$
\end{lemma}
\begin{proof}
It follows from a straightforward verification.
\end{proof}

For $i,j\in\mathbb{Z}_n$, let $S_{i,j}$ be the one dimensional $\mathcal{H}_n(q)$-module
defined by $bv=q^iv$, $cv=q^jv$ and $av=dv=0$, $v\in S_{i,j}$.
Let $P_{i,j}=P(S_{i,j})$ be the projective cover of $S_{i,j}$.
Let $J={\rm rad}(\mathcal{H}_n(q))$ be the Jacobson radical of $\mathcal{H}_n(q)$.

\begin{lemma}\label{3.3}
The simple modules $S_{i,j}$, $i,j\in\mathbb{Z}_n$, exhaust all simple modules of $\mathcal{H}_n(q)$,
and consequently, the projective modules $P_{i,j}$, $i,j\in\mathbb{Z}_n$, exhaust all
indecomposable projective modules of $\mathcal{H}_n(q)$. Moreover,
$P_{i,j}\cong\mathcal{H}_n(q)e_{i,j}$ for all $i,j\in\mathbb{Z}_n$.
\end{lemma}

\begin{proof}
Obviously, $a\mathcal{H}_n(q)=\mathcal{H}_n(q)a$ and $d\mathcal{H}_n(q)=\mathcal{H}_n(q)d$.
Since $a^n=0$ and $d^n=0$, $\mathcal{H}_n(q)a+\mathcal{H}_n(q)d$ is a nilpotent ideal of
$\mathcal{H}_n(q)$. Hence $\mathcal{H}_n(q)a+\mathcal{H}_n(q)d\subseteq J$.
On the other hand, it is easy to see that the quotient algebra
$\mathcal{H}_n(q)/(\mathcal{H}_n(q)a+\mathcal{H}_n(q)d)$ is isomorphic to the group
algebra $\mathbb{K}G$, where $G=G(\mathcal{H}_n(q))=\{b^ic^j|0\leqslant i, j\leqslant n-1\}$,
the group of all group-like elements of $\mathcal{H}_n(q)$.
Since $\mathbb{K}G$ is semisimple, $J\subseteq \mathcal{H}_n(q)a+\mathcal{H}_n(q)d$.
Thus, $J=\mathcal{H}_n(q)a+\mathcal{H}_n(q)d$.
Therefore, the simple modules $S_{i,j}$ exhaust all simple modules of $\mathcal{H}_n(q)$,
and the projective modules $P_{i,j}$ exhaust all indecomposable projective modules of
$\mathcal{H}_n(q)$, $i,j\in\mathbb{Z}_n$. The last statement of the lemma
follows from Lemma \ref{3.2}.
\end{proof}

\begin{corollary}\label{3.4}
$\mathcal{H}_n(q)$ is a basic algebra. Moreover,
$J$ is a Hopf ideal of $\mathcal{H}_n(q)$, and the Loewy length of $\mathcal{H}_n(q)$ is $2n-1$.

\end{corollary}
\begin{proof}
It follows from Lemma \ref{3.3} that $\mathcal{H}_n(q)$ is a basic algebra.
By $J=\mathcal{H}_n(q)a+\mathcal{H}_n(q)d$, one can easily check that
$J$ is a coideal and $S(J)\subseteq J$. Hence $J$ is a Hopf ideal.
By $a^{n-1}\neq 0$ and $d^{n-1}\neq 0$, one gets $(\mathcal{H}_n(q)a+\mathcal{H}_n(q)d)^{2n-2}\neq0$.
By $a^n=d^n=0$, one gets $(\mathcal{H}_n(q)a+\mathcal{H}_n(q)d)^{2n-1}=0$.
It follows that the Loewy length of $\mathcal{H}_n(q)$ is $2n-1$.
\end{proof}

In the rest of this section, we regard that $P_{i,j}=\mathcal{H}_n(q)e_{i,j}$ for all $i,j\in\mathbb{Z}_n$.

\begin{corollary}\label{3.5}
$P_{i,j}$ is $n^2$-dimensional with a $\mathbb K$-basis $\{a^kd^le_{i,j}|0\leqslant k, l\leqslant n-1\}$,
$i,j\in\mathbb{Z}_n$. Consequently, $\mathcal{H}_n(q)$ is an
indecomposable algebra.
\end{corollary}

\begin{proof}
By Lemma \ref{3.2}, $P_{i,j}={\rm span}\{a^kd^le_{i,j}|0\leqslant k, l\leqslant n-1\}$, and hence
dim$P_{i,j}\leqslant n^2$. Now it follows from
$\mathcal{H}_n(q)=\oplus_{i,j\in\mathbb{Z}_n}\mathcal{H}_n(q)e_{i,j}$
and dim$\mathcal{H}_n(q)=n^4$ that $P_{i,j}$ is $n^2$-dimensional over $\mathbb{K}$ with a basis
$\{a^kd^le_{i,j}|0\leqslant k, l\leqslant n-1\}$. Then by Lemmas \ref{3.2}-\ref{3.3}, one knows that
every simple module is a simple factor of $P_{i,j}$ with the multiplicity one.
Consequently, $\mathcal{H}_n(q)$ is an indecomposable algebra.
\end{proof}

Given $M\in\mathcal{H}_n(q)$-mod, for any $\a\in\mathbb{K}$ and $u,v\in M$,
we use $u\xrightarrow{\a}v$ (resp. $u\stackrel{\a}{\dashrightarrow}v$) to represent $a\cdot u=\a v$
(resp. $d\cdot u=\a v$). Moreover, we omit the decoration of the arrow if $\a=1$.

For $i, j\in\mathbb{Z}_n$, let $e_{i,j}^{k,l}=a^kd^le_{i,j}$ in $P_{i,j}$, $0\<k,l\<n-1$.
Then the structure of $P_{i,j}$ can be described as follows:
$$
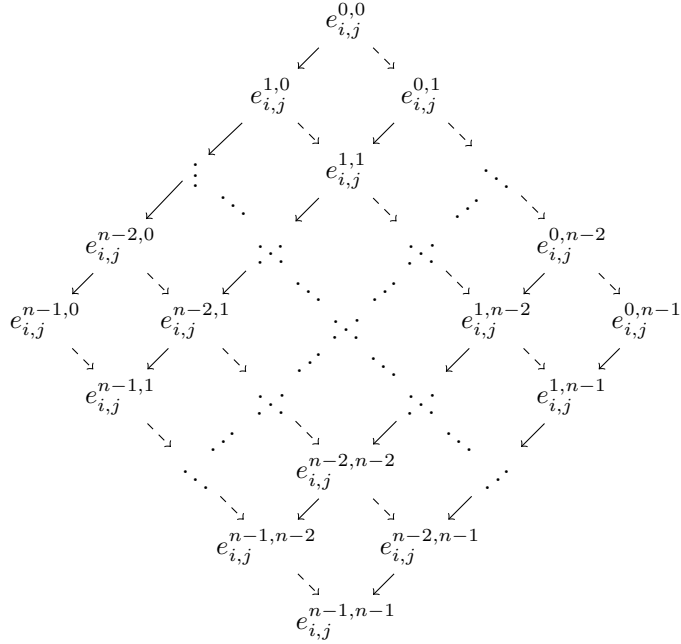
\begin{tikzpicture}[scale=1]
\path (0,0) node(e) {$e_{i,j}^{0,0}$};
\path (-1,-1) node(a) {$e_{i,j}^{1,0}$} (1,-1) node(d) {$e_{i,j}^{0,1}$};
\path (-2,-1.95) node(a2) {$\vdots$}
(0,-2) node(ad) {$e_{i,j}^{1,1}$} (2,-2) node(d2) {$\ddots$};
\path (-1.5, -2.35) node(1) {$\ddots$} (1.6, -2.4) node(1c) {$\adots$};
\path (-3,-3) node(an-2) {$e_{i,j}^{n-2,0}$}
(-1.14,-3.23) node(2l) {$\cdot$} (-1,-3) node(2) {$\ddots$} (-0.87,-3.01) node(2r) {$\cdot$}
(0.85,-3.235) node(3l) {$\cdot$} (1,-3) node(3) {$\ddots$}
(1.15,-3) node(3r) {$\cdot$} (3,-3) node (dn-2) {$e_{i,j}^{0,n-2}$};
\path (-0.5,-3.5) node(3.5l) {$\ddots$} (0.5,-3.5) node(3.5c) {$\adots$};
\path (-4,-4) node(an-1) {$e_{i,j}^{n-1,0}$} (-2, -4) node(an-2d) {$e_{i,j}^{n-2,1}$}
(-0.15,-4.22) node(4l) {$\cdot$} (0,-4) node(4) {$\ddots$} (0.15,-4) node(4r) {$\cdot$}
(2,-4) node(adn-2) {$e_{i,j}^{1,n-2}$} (4, -4) node(dn-1) {$e_{i,j}^{0,n-1}$};
\path (-0.5,-4.45) node(4.5l) {$\adots$} (0.5,-4.5) node(4.5r) {$\ddots$};
\path (-3,-5) node(an-1d) {$e_{i,j}^{n-1,1}$}
(-1.15,-5.235) node(5l) {$\cdot$} (-1,-5) node(5) {$\ddots$} (-0.85,-5) node(5r) {$\cdot$}
(0.85,-5.235) node(6l) {$\cdot$} (1,-5) node(6) {$\ddots$}
(1.15, -5) node(6r) {$\cdot$} (3,-5) node (adn-1) {$e_{i,j}^{1,n-1}$};
\path (-1.65,-5.5) node(5.5l) {$\adots$} (1.5, -5.5) node(5.5r) {$\ddots$};
\path (-2,-6) node(7) {$\ddots$} (0,-6) node(an-2dn-2) {$e_{i,j}^{n-2,n-2}$}
(2,-6) node (9) {$\adots$};
\path (-1,-7) node(an-1dn-2) {$e_{i,j}^{n-1,n-2}$\ \ }
(1,-7) node(an-2dn-1) {\ \ $e_{i,j}^{n-2,n-1}$};
\path (0,-8) node(an-1dn-1) {$e_{i,j}^{n-1,n-1}$};
\draw[->] (e) --(a);
\draw[->,dashed] (e) --(d);
\draw[->] (a) --(a2);
\draw[->,dashed] (a) --(ad);
\draw[->] (d) --(ad);
\draw[->,dashed] (d) --(d2);
\draw[->] (a2) --(an-2);
\draw[->] (ad) --(2);
\draw[->,dashed] (ad) --(3);
\draw[->,dashed] (d2) --(dn-2);
\draw[->] (an-2) --(an-1);
\draw[->,dashed] (an-2) --(an-2d);
\draw[->] (2) --(an-2d);
\draw[->,dashed] (3) --(adn-2);
\draw[->] (dn-2) --(adn-2);
\draw[->,dashed] (dn-2) --(dn-1);
\draw[->,dashed] (an-1) --(an-1d);
\draw[->] (an-2d) --(an-1d);
\draw[->,dashed] (an-2d) --(5);
\draw[->] (adn-2) --(6);
\draw[->,dashed] (adn-2) --(adn-1);
\draw[->] (dn-1) --(adn-1);
\draw[->,dashed] (an-1d) --(7);
\draw[->,dashed] (5) --(an-2dn-2);
\draw[->] (6) --(an-2dn-2);
\draw[->] (adn-1) --(9);
\draw[->,dashed] (7) --(an-1dn-2);
\draw[->] (an-2dn-2) --(an-1dn-2);
\draw[->,dashed] (an-2dn-2) --(an-2dn-1);
\draw[->] (9) --(an-2dn-1);
\draw[->,dashed] (an-1dn-2) --(an-1dn-1);
\draw[->] (an-2dn-1) --(an-1dn-1);
\end{tikzpicture}$$

\begin{proposition}\label{3.6}
$S_{i,j}\ot S_{k,l}\cong S_{i+k,j+l}$ and
$S_{i,j}\ot P_{k,l}\cong P_{k,l}\ot S_{i,j}\cong P_{i+k,j+l}$
for all $i,j,k,l\in\mathbb{Z}_n$.
\end{proposition}

\begin{proof}
The first isomorphism is obvious. Note that $S_{0,0}$ is the trivial $\mathcal{H}_n(q)$-module.
Since $J$ is a Hopf ideal, it follows from \cite[Corollary 3.3]{Lo} and the first isomorphism that
$P_{k,l}\ot S_{i,j}\cong P_{0,0}\ot S_{k,l}\ot S_{i,j}\cong P_{0,0}\ot S_{i+k,j+l}\cong P_{i+k,j+l}$.
Similarly, one can show that $S_{i,j}\ot P_{k,l}\cong P_{i+k, j+l}$,
which also follows from the proof of \cite[Lemma 3.3]{Cib99}.
\end{proof}

\begin{proposition}\label{3.7}
Let $i,j,k,l\in\mathbb{Z}_n$. Then $P_{i,j}\ot P_{k,l}\cong \oplus_{r,t\in\mathbb{Z}_n}P_{r,t}$.
\end{proposition}

\begin{proof}
By Proposition \ref{3.6}, we only need to consider the case of $i=j=k=l=0$.
For any short exact sequence $0\ra N\ra M\ra L\ra 0$ of modules,
the exact sequence $0\ra P_{0,0}\ot N\ra P_{0,0}\ot M\ra P_{0,0}\ot L\ra 0$
is always split since $P_{0,0}\ot L$ is projective
for any module $L$. By Corollary \ref{3.4} and the proof of Corollary \ref{3.5},
$[P_{0,0}]=\sum_{r,t\in\mathbb{Z}_n}[S_{r,t}]$ in $G_0(\mathcal{H}_n(q))$.
Then it follows from Proposition \ref{3.6} that
$P_{0,0}\ot P_{0,0}\cong\oplus_{r,t\in\mathbb{Z}_n}P_{0,0}\ot S_{r,t}\cong \oplus_{r,t\in\mathbb{Z}_n}P_{r,t}$,
which is isomorphic to the regular module $\mathcal{H}_n(q)$.
\end{proof}

By Propositions \ref{3.6} and \ref{3.7}, the projective class ring $r_p(\mathcal{H}_n(q))$
is a commutative ring generated by $[S_{1,0}]$,
$[S_{0,1}]$ and $[P_{0,0}]$ subject to the relations
$[S_{1,0}]^n=1$, $[S_{0,1}]^n=1$ and $[P_{0,0}]^2=\sum_{i,j=0}^{n-1}[S_{1,0}]^i[S_{0,1}]^j[P_{0,0}]$.
Hence we have the following proposition.

\begin{theorem}\label{3.8}
$r_p(\mathcal{H}_n(q))\cong \mathbb{Z}[x,y,z]/(x^n-1,y^n-1,z^2-\sum_{i,j=0}^{n-1}x^iy^jz)$.
\end{theorem}

\begin{proof}
By Propositions \ref{3.6} and \ref{3.7}, $r_p(\mathcal{H}_n(q))$ is a commutative ring.
Moreover, $r_p(\mathcal{H}_n(q))$ is generated, as a $\mathbb Z$-algebra, by
$[S_{1,0}]$, $[S_{0,1}]$ and $[P_{0,0}]$. Therefore, there exists a ring epimorphism
$\phi: \mathbb{Z}[x,y,z]\ra r_p(\mathcal{H}_n(q))$ such that
$\phi(x)=[S_{1,0}]$, $\phi(y)=[S_{0,1}]$ and $\phi(z)=[P_{0,0}]$.
Let $I=(x^n-1,y^n-1,z^2-\sum_{i,j=0}^{n-1}x^iy^jz)$ be the ideal of $\mathbb{Z}[x,y,z]$
generated by $x^n-1$, $y^n-1$ and $z^2-\sum_{i,j=0}^{n-1}x^iy^jz$.
Then it follows from Propositions \ref{3.6} and \ref{3.7} that
$I\subseteq{\rm Ker}(\phi)$. Hence $\phi$ induces a ring epimorphism
$\ol{\phi}: \mathbb{Z}[x,y,z]/I\ra r_p(\mathcal{H}_n(q))$ such that
$\ol{\phi}\circ\pi=\phi$, where $\pi: \mathbb{Z}[x,y,z]\ra \mathbb{Z}[x,y,z]/I$
is the canonical projection. Let $\ol{u}=\pi(u)$ for any $u\in\mathbb{Z}[x,y,z]$.
Then $\ol{x}^n=1$, $\ol{y}^n=1$ and $\ol{z}^2=\sum_{i,j=0}^{n-1}\ol{x}^i\ol{y}^j\ol{z}$
in $\mathbb{Z}[x,y,z]/I$. Hence $\mathbb{Z}[x,y,z]/I$ is generated,
as a $\mathbb Z$-module, by
$\{\ol{x}^i\ol{y}^j, \ol{x}^i\ol{y}^j\ol{z}|i,j\in\mathbb{Z}_n\}$.
Since $r_p(\mathcal{H}_n(q))$ is a free $\mathbb Z$-module with a $\mathbb Z$-basis
$\{[S_{i,j}], [P_{i,j}]|i,j\in\mathbb{Z}_n\}$, one can define a $\mathbb Z$-module
map $\psi: r_p(\mathcal{H}_n(q))\ra\mathbb{Z}[x,y,z]/I$ by
$\psi([S_{i,j}])=\ol{x}^i\ol{y}^{j}$ and $\psi([P_{i,j}])=\ol{x}^i\ol{y}^{j}\ol{z}$
for any $i,j\in\mathbb{Z}_n$. Now for any $i,j\in\mathbb{Z}_n$, we have
$\psi(\ol{\phi}(\ol{x}^i\ol{y}^j))=\psi(\ol{\phi}(\ol{x})^i\ol{\phi}(\ol{y})^j)
=\psi([S_{1,0}]^i[S_{0,1}]^j)=\psi([S_{i,j}])=\ol{x}^i\ol{y}^j$ and
$\psi(\ol{\phi}(\ol{x}^i\ol{y}^j\ol{z}))=\psi(\ol{\phi}(\ol{x})^i\ol{\phi}(\ol{y})^j\ol{\phi}(\ol{z}))
=\psi([S_{1,0}]^i[S_{0,1}]^j[P_{0,0}])=\psi([P_{i,j}])=\ol{x}^i\ol{y}^j\ol{z}$.
This shows that $\ol{\phi}$ is injective, and so
$\ol{\phi}$ is a ring isomorphism.
\end{proof}

Now we consider the projective class algebra $R_p(\mathcal{H}_n(q))$.
By Theorem \ref{3.8}, we have
$$\begin{array}{c}
R_p(\mathcal{H}_n(q))\cong \mathbb{K}[x,y,z]/(x^n-1,y^n-1,z^2-\sum_{i,j=0}^{n-1}x^iy^jz).\\
\end{array}$$
Put $I=(x^n-1,y^n-1,z^2-\sum_{i,j=0}^{n-1}x^iy^jz)$ and let $J( \mathbb{K}[x,y,z]/I)$
be the Jacobson radical of $\mathbb{K}[x,y,z]/I$.
For any $u\in\mathbb{K}[x,y,z]$, let $\ol{u}$ denote the image of $u$
under the canonical projection $\mathbb{K}[x,y,z]\ra\mathbb{K}[x,y,z]/I$.
Then by the proof of Proposition \ref{3.8}, $\mathbb{K}[x,y,z]/I$ is of
dimension $2n^2$ with a $\mathbb K$-basis
$\{\ol{x}^i\ol{y}^j, \ol{x}^i\ol{y}^j\ol{z}|0\<i,j\<n-1\}$.
From $\ol{x}^n=1$, $\ol{y}^n=1$ and $\ol{z}^2=\sum_{i,j=0}^{n-1}\ol{x}^i\ol{y}^j\ol{z}$,
one gets $(1-\ol{x})\ol{z}^2=(1-\ol{y})\ol{z}^2=0$, and so $((1-\ol{x})\ol{z})^2=((1-\ol{y})\ol{z})^2=0$.
Consequently, the ideal $((1-\ol{x})\ol{z}, (1-\ol{y})\ol{z})$ of $\mathbb{K}[x,y,z]/I$
generated by $(1-\ol{x})\ol{z}$ and $(1-\ol{y})\ol{z}$ is contained in $J( \mathbb{K}[x,y,z]/I)$.
Moreover, dim$((\mathbb{K}[x,y,z]/I)/((1-\ol{x})\ol{z}, (1-\ol{y})\ol{z})=n^2+1$ and
$$\begin{array}{rl}
&(\mathbb{K}[x,y,z]/I)/((1-\ol{x})\ol{z}, (1-\ol{y})\ol{z})\\
\cong&\mathbb{K}[x,y,z]/(x^n-1,y^n-1,z^2-n^2z, (1-x)z, (1-y)z).\\
\end{array}$$
Let $\pi: \mathbb{K}[x,y,z]\ra \mathbb{K}[x,y,z]/(x^n-1,y^n-1,z^2-n^2z, (1-x)z, (1-y)z)$
be the canonical projection. For any integers $k,l\>0$, let $f_{k,l}=\frac{1}{n^2}\sum_{i,j=0}^{n-1}q^{ki+lj}x^iy^j$
in $\mathbb{K}[x,y,z]$. Then a straightforward verification shows that
$$\begin{array}{c}
\{\pi(f_{k,l}), \pi(f_{0,k}), \pi(f_{0,0}-\frac{1}{n^2}z), \pi(\frac{1}{n^2}z)|1\<k\<n-1, 0\<l\<n-1\}\\
\end{array}$$
is a set of orthogonal idempotents, and so it is a full set of orthogonal primitive idempotents
in $\mathbb{K}[x,y,z]/(x^n-1,y^n-1,z^2-n^2z, (1-x)z, (1-y)z)$. Therefore,
$$\mathbb{K}[x,y,z]/(x^n-1,y^n-1,z^2-n^2z, (1-x)z, (1-y)z)\cong\mathbb{K}^{n^2+1}.$$
Thus, $J( \mathbb{K}[x,y,z]/I)\subseteq((1-\ol{x})\ol{z}, (1-\ol{y})\ol{z})$,
and so $J( \mathbb{K}[x,y,z]/I)=((1-\ol{x})\ol{z}, (1-\ol{y})\ol{z})$.
This shows the following proposition.

\begin{proposition}\label{3.9}
Let $J(R_p(\mathcal{H}_n(q)))$ be the Jacobson radical of $R_p(\mathcal{H}_n(q))$.
Then $J(R_p(\mathcal{H}_n(q)))=((1-[S_{1,0}])[P_{0,0}], (1-[S_{0,1}])[P_{0,0}])$ and
$$\begin{array}{rl}
&R_p(\mathcal{H}_n(q))/J(R_p(\mathcal{H}_n(q)))\\
\cong&\mathbb{K}[x,y,z]/(x^n-1,y^n-1,z^2-n^2z, (1-x)z, (1-y)z)
\cong{\mathbb K}^{n^2+1}.\\
\end{array}$$
\end{proposition}

\section{\bf The Projective Class Ring of $H_n(0,q)$}\selabel{4}

In this section, we investigate the projective class ring of $H_n(0,q)$.

\begin{proposition}\label{4.1}
$H_n(0,q)$ is a symmetric algebra.
\end{proposition}

\begin{proof}
By \cite[Proposition 3.4]{Ch1} and its proof, $H_n(0,q)$ is unimodular.
Moreover, $S^2(a)=qa$, $S^2(b)=b$, $S^2(c)=c$ and $S^2(d)=q^{-1}d$,
where $S$ is the antipode of $H_n(0, q)$.
Hence $S^2(x)=bxb^{-1}=cxc^{-1}$ for all $x\in H_n(0,q)$. That is,
$S^2$ is an inner automorphism of $H_n(0,q)$.
It follows from \cite{Lo, ObSch} that $H_n(0,q)$
is a symmetric algebra.
\end{proof}

Note that $\mathcal{H}_n(q)$ is not symmetric since it is not unimodular.

\begin{proposition}\label{4.1+1}
$H_n(0, q)$ is of wild representation type.
\end{proposition}

\begin{proof}
It is similar to Proposition \ref{3.1}.
Let $A$ be the subalgebra of $H_n(0, q)$ generated by $a$ and $d$.
Then $A$ is a $\mathbb{K}G$-module algebra with the action given by
$b\cdot a=qa$, $b\cdot d=q^{-1}d$, $c\cdot a=qa$ and $c\cdot d=q^{-1}d$,
 where $G=G(H_n(0,q))=\{b^ic^j|i,j\in\mathbb{Z}_n\}\cong \mathbb{Z}_n\times\mathbb{Z}_n$.
Moreover, $A\cong\mathbb{K}\langle x, y\rangle/(x^n, y^n, yx-qxy)$ and $H_n(0,q)\cong A\#\mathbb{K}G$,
as $\mathbb{K}$-algebras. Since $n\>3$, it follows from \cite[p.295(3.4)]{Ringel} that $A$
is of wild representation type. Since $\mathbb{K}G$ is a semisimple and cosemisimple Hopf algebra
by  char$(\mathbb{K})\nmid |G|$, it follows from
\cite[Theorem 4.5]{Liu} that $A\#\mathbb{K}G$ is of wild representation type.
\end{proof}

$H_n(0,q)$ has $n^2$ orthogonal primitive idempotents
$$\begin{array}{c}
e_{i,j}=\frac{1}{n^2}\sum_{k,l\in\mathbb{Z}_n}q^{-ik-jl}b^kc^l
=\frac{1}{n^2}\sum_{k,l=0}^{n-1}q^{-ik-jl}b^kc^l,\ \ i,j\in\mathbb{Z}_n.\\
\end{array}$$

\begin{lemma}\label{4.2}
Let $i,j\in\mathbb{Z}_n$. Then
$$be_{i,j}=q^ie_{i,j}, \ ce_{i,j}=q^je_{i,j},\
ae_{i,j}=e_{i+1,j+1}a, \ de_{i,j}=e_{i-1,j-1}d.$$
\end{lemma}
\begin{proof}
It follows from a straightforward verification.
\end{proof}

For $i,j\in\mathbb{Z}_n$, let $S_{i,j}$ be the one dimensional $H_n(0,q)$-module
defined by $bv=q^iv$, $cv=q^jv$ and $av=dv=0$, $v\in S_{i,j}$.
Let $P_{i,j}=P(S_{i,j})$ be the projective cover of $S_{i,j}$.
Let $J={\rm rad}(H_n(0,q))$ be the Jacobson radical of $H_n(0,q)$.

\begin{lemma}\label{4.3}
The simple modules $S_{i,j}$, $i,j\in\mathbb{Z}_n$, exhaust all simple modules of $H_n(0,q)$,
and consequently, the projective modules $P_{i,j}$, $i,j\in\mathbb{Z}_n$, exhaust all
indecomposable projective modules of $H_n(0,q)$. Moreover,
$P_{i,j}\cong H_n(0,q)e_{i,j}$ for all $i,j\in\mathbb{Z}_n$.
\end{lemma}

\begin{proof}
It is similar to Lemma \ref{3.3}.
\end{proof}

\begin{corollary}\label{4.4}
$H_n(0, q)$ is a basic algebra. Moreover,
$J$ is a Hopf ideal of $H_n(0,q)$, and the Loewy length of $H_n(0,q)$ is $2n-1$.

\end{corollary}
\begin{proof}
It is similar to Corollary \ref{3.4}.
\end{proof}

Let $e_i=\sum_{j=0}^{n-1}e_{i+j,j}=\frac{1}{n}\sum_{j=0}^{n-1}q^{-ij}b^jc^{-j}$,
$i\in\mathbb{Z}_n$. Then by Lemmas \ref{4.2} and \ref{4.3},
$\{e_i|i\in\mathbb{Z}_n\}$ is a full set of central primitive idempotents
of $H_n(0,q)$. Hence $H_n(0,q)$ decomposes into $n$ blocks
$H_n(0, q)e_i$, $i\in\mathbb{Z}_n$.

In the rest of this section, we regard that $P_{i,j}=H_n(0,q)e_{i,j}$ for all $i,j\in\mathbb{Z}_n$.

\begin{corollary}\label{4.5}
$P_{i,j}$ is $n^2$-dimensional with a $\mathbb K$-basis $\{a^kd^le_{i,j}|0\leqslant k, l\leqslant n-1\}$,
$i,j\in\mathbb{Z}_n$.
\end{corollary}

\begin{proof}
It is similar to Corollary \ref{3.5}.
\end{proof}

For $i, j\in\mathbb{Z}_n$, let $e_{i,j}^{k,l}=a^kd^le_{i,j}$ in $P_{i,j}$.
Using the same symbols as in the last section,
the structure of $P_{i,j}$ can be described as follows:
$$
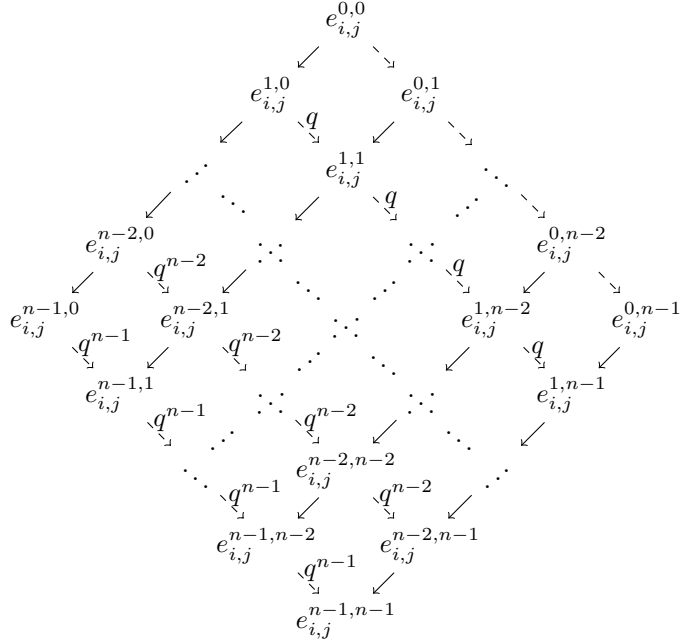
\begin{tikzpicture}[scale=1]
\path (0,0) node(e) {$e_{i,j}^{0,0}$};
\path (-1,-1) node(a) {$e_{i,j}^{1,0}$} (1,-1) node(d) {$e_{i,j}^{0,1}$};
\path (-0.45,-1.35) node(q1) {$q$};
\path (-2,-1.95) node(a2) {$\adots$}
(0,-2) node(ad) {$e_{i,j}^{1,1}$} (2,-2) node(d2) {$\ddots$};
\path (0.6,-2.4) node(q2) {$q$};
\path (-1.5, -2.35) node(1) {$\ddots$} (1.6, -2.4) node(1c) {$\adots$};
\path (-3,-3) node(an-2) {$e_{i,j}^{n-2,0}$}
(-1.14,-3.23) node(2l) {$\cdot$} (-1,-3) node(2) {$\ddots$} (-0.87,-3.01) node(2r) {$\cdot$}
(0.85,-3.235) node(3l) {$\cdot$} (1,-3) node(3) {$\ddots$}
(1.15,-3) node(3r) {$\cdot$} (3,-3) node (dn-2) {$e_{i,j}^{0,n-2}$};
\path (-2.2,-3.3) node(q3) {$q^{n-2}$} (-0.5,-3.5) node(3.5l) {$\ddots$}
(0.5,-3.5) node(3.5c) {$\adots$} (1.5,-3.3) node(q4) {$q$};
\path (-4,-4) node(an-1) {$e_{i,j}^{n-1,0}$} (-2, -4) node(an-2d) {$e_{i,j}^{n-2,1}$}
(-0.15,-4.22) node(4l) {$\cdot$} (0,-4) node(4) {$\ddots$} (0.15,-4) node(4r) {$\cdot$}
(2,-4) node(adn-2) {$e_{i,j}^{1,n-2}$} (4, -4) node(dn-1) {$e_{i,j}^{0,n-1}$};
\path (-3.2,-4.3) node(q5) {$q^{n-1}$} (-1.2,-4.3) node(q6) {$q^{n-2}$}
(-0.5,-4.45) node(4.5l) {$\adots$} (0.5,-4.5) node(4.5r) {$\ddots$}
(2.55,-4.4) node(q7) {$q$};
\path (-3,-5) node(an-1d) {$e_{i,j}^{n-1,1}$}
(-1.15,-5.235) node(5l) {$\cdot$} (-1,-5) node(5) {$\ddots$} (-0.85,-5) node(5r) {$\cdot$}
(0.85,-5.235) node(6l) {$\cdot$} (1,-5) node(6) {$\ddots$}
(1.15, -5) node(6r) {$\cdot$} (3,-5) node (adn-1) {$e_{i,j}^{1,n-1}$};
\path (-2.2,-5.3) node(q8) {$q^{n-1}$} (-1.65,-5.5) node(5.5l) {$\adots$}
(-0.2,-5.3) node(q9) {$q^{n-2}$} (1.5, -5.5) node(5.5r) {$\ddots$};
\path (-2,-6) node(7) {$\ddots$} (0,-6) node(an-2dn-2) {$e_{i,j}^{n-2,n-2}$}
(2,-6) node (9) {$\adots$};
\path (-1.2,-6.3) node(q10) {$q^{n-1}$} (0.8,-6.3) node(q11) {$q^{n-2}$};
\path (-1,-7) node(an-1dn-2) {$e_{i,j}^{n-1,n-2}$\ \ }
(1,-7) node(an-2dn-1) {\ \ $e_{i,j}^{n-2,n-1}$};
\path (-0.2,-7.3) node(q12) {$q^{n-1}$};
\path (0,-8) node(an-1dn-1) {$e_{i,j}^{n-1,n-1}$};
\draw[->] (e) --(a);
\draw[->,dashed] (e) --(d);
\draw[->] (a) --(a2);
\draw[->,dashed] (a) --(ad);
\draw[->] (d) --(ad);
\draw[->,dashed] (d) --(d2);
\draw[->] (a2) --(an-2);
\draw[->] (ad) --(2);
\draw[->,dashed] (ad) --(3);
\draw[->,dashed] (d2) --(dn-2);
\draw[->] (an-2) --(an-1);
\draw[->,dashed] (an-2) --(an-2d);
\draw[->] (2) --(an-2d);
\draw[->,dashed] (3) --(adn-2);
\draw[->] (dn-2) --(adn-2);
\draw[->,dashed] (dn-2) --(dn-1);
\draw[->,dashed] (an-1) --(an-1d);
\draw[->] (an-2d) --(an-1d);
\draw[->,dashed] (an-2d) --(5);
\draw[->] (adn-2) --(6);
\draw[->,dashed] (adn-2) --(adn-1);
\draw[->] (dn-1) --(adn-1);
\draw[->,dashed] (an-1d) --(7);
\draw[->,dashed] (5) --(an-2dn-2);
\draw[->] (6) --(an-2dn-2);
\draw[->] (adn-1) --(9);
\draw[->,dashed] (7) --(an-1dn-2);
\draw[->] (an-2dn-2) --(an-1dn-2);
\draw[->,dashed] (an-2dn-2) --(an-2dn-1);
\draw[->] (9) --(an-2dn-1);
\draw[->,dashed] (an-1dn-2) --(an-1dn-1);
\draw[->] (an-2dn-1) --(an-1dn-1);
\end{tikzpicture}$$

\begin{proposition}\label{4.6}
The $n$ blocks $H_n(0,q)e_i$, $i\in\mathbb{Z}_n$,
are isomorphic to each other.
\end{proposition}

\begin{proof}
Let $i\in\mathbb{Z}_n$. Since $e_i=\sum_{j=0}^{n-1}e_{i+j,j}$,
$H_n(0,q)e_i=\oplus_{j=0}^{n-1}H_n(0,q)e_{i+j,j}$
as $H_n(0,q)$-modules. Then by Corollary \ref{4.5},
dim$(H_n(0,q)e_i)=n^3$. By Lemma \ref{4.2}, one gets
$be_i=q^ice_i$. It follows that
$H_n(0,q)e_i={\rm span}\{a^jd^kb^le_i|0\leqslant j,k,l\leqslant n-1\}$,
and so $\{a^jd^kb^le_i|0\leqslant j,k,l\leqslant n-1\}$ is a $\mathbb K$-basis
of $H_n(0, q)e_i$. Let $B$ be the subalgebra of $H_n(q)$ generated by
$a$, $b$ and $d$. Then one can easily check that
the block $H_n(0,q)e_i$ is isomorphic, as an algebra,
to the subalgebra $B$ of $H_n(0,q)$. Thus, the proposition follows.
\end{proof}

Let $i\in\mathbb{Z}_n$ be fixed. For any $j\in\mathbb{Z}_n$,
let $\ol{e}_j=e_{i+j,j}$. Then the Gabriel quiver
$Q=(Q_0, Q_1)$
of the block $H_n(0,q)e_i$ is given by
$$
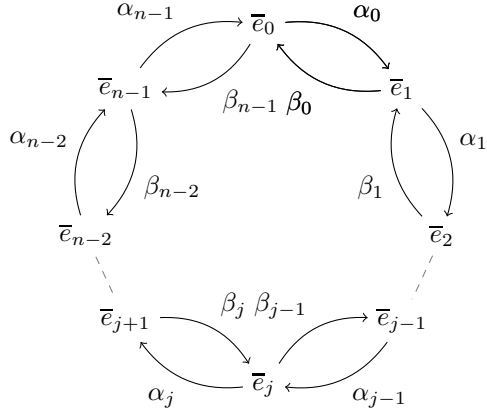
\begin{tikzpicture}[auto,bend right,scale=0.6]
\node (e0) at (90:4) {$\ol{e}_0$};
\node (e1) at (40:4) {$\ol{e}_1$};
\node (e2) at (-10:4) {$\ol{e}_2$};
\node (en_1) at (140:4) {$\ol{e}_{n-1}$};
\node (en_2) at (190:4) {$\ol{e}_{n-2}$};
\node (ejplus1) at (220:4) {$\ol{e}_{j+1}$};
\node (ej) at (270:4) {$\ol{e}_j$};
\node (ej_1) at (320:4) {$\ol{e}_{j-1}$};
\draw [<-](e0) to node [swap] {$\beta_0$} (e1);
\draw [<-](e1) to node [swap] {$\alpha_0$} (e0);
\draw [<-](e1) to node [swap] {$\beta_1$} (e2);
\draw [<-](e2) to node [swap] {$\alpha_1$} (e1);
\draw [<-](en_1) to node [swap] {$\beta_{n-1}$} (e0);
\draw [<-](e0) to node [swap] {$\alpha_{n-1}$} (en_1);
\draw [<-](en_2) to node [swap] {$\beta_{n-2}$} (en_1);
\draw [<-](en_1) to node [swap] {$\alpha_{n-2}$} (en_2);
\draw [<-](e0) to node [swap] {$\beta_0$} (e1);
\draw [<-](e1) to node [swap] {$\alpha_0$} (e0);
\draw [<-](ej) to node [swap] {$\alpha_{j-1}$} (ej_1);
\draw [<-](ej_1) to node [swap] {$\beta_{j-1}$} (ej);
\draw [<-](ejplus1) to node [swap] {$\alpha_j$} (ej);
\draw [<-](ej) to node [swap] {$\beta_j$} (ejplus1);
\draw[gray,dashed] (e2) -- (ej_1);
\draw[gray,dashed] (en_2) -- (ejplus1);
\end{tikzpicture}$$
where for $j\in\mathbb{Z}_n$, the arrows $\a_j$, $\b_j$ correspond to $a\ol{e}_j$,
$d\ol{e}_{j+1}$, respectively. The admissible ideal $I$ has the following relations:
$$\b_{j}\a_{j}-q\a_{j-1}\b_{j-1}=0,\ \a_{j+(n-1)}\cdots\a_{j+1}\a_{j}=0,\
\b_{j-(n-1)}\cdots\b_{j-1}\b_{j}=0,\ j\in\mathbb{Z}_n.$$

\begin{proposition}\label{4.7}
$S_{i,j}\ot S_{k,l}\cong S_{i+k,j+l}$ and
$S_{i,j}\ot P_{k,l}\cong P_{k,l}\ot S_{i,j}\cong P_{i+k,j+l}$
for all $i,j,k,l\in\mathbb{Z}_n$.
\end{proposition}

\begin{proof}
It is similar to Proposition \ref{3.6}.
\end{proof}

\begin{proposition}\label{4.8}
Let $i,j,k,l\in\mathbb{Z}_n$. Then $P_{i,j}\ot P_{k,l}\cong \oplus_{t\in\mathbb{Z}_n}nP_{i+k+t,j+l+t}$.
\end{proposition}

\begin{proof}
It is similar to Proposition \ref{3.7}. Note that $[P_{0,0}]=\sum_{t=0}^{n-1}n[S_{t,t}]$
in $G_0(H_n(0,q))$ by Corollaries \ref{4.4} and \ref{4.5}.
\end{proof}

\begin{theorem}\label{4.9}
$r_p(H_n(0,q))\cong \mathbb{Z}[x,y,z]/(x^n-1,y^n-1,z^2-n\sum_{i=0}^{n-1}x^iz)$.
\end{theorem}

\begin{proof}
It is similar to Theorem \ref{3.8}. Note that $r_p(H_n(0,q))$ is a commutative ring
generated by $[S_{1,1}]$, $[S_{0,1}]$ and $[P_{0,0}]$.
\end{proof}

Now we consider the projective class algebra $R_p(H_n(0,q))$.
By Theorem \ref{4.9}, we have
$$\begin{array}{c}
R_p(H_n(0,q))\cong \mathbb{K}[x,y,z]/(x^n-1,y^n-1,z^2-n\sum_{i=0}^{n-1}x^iz).\\
\end{array}$$
Put $I=(x^n-1,y^n-1,z^2-n\sum_{i=0}^{n-1}x^iz)$ and let $J( \mathbb{K}[x,y,z]/I)$
be the Jacobson radical of $\mathbb{K}[x,y,z]/I$.
For any $u\in\mathbb{K}[x,y,z]$, let $\ol{u}$ denote the image of $u$
under the canonical projection $\mathbb{K}[x,y,z]\ra\mathbb{K}[x,y,z]/I$.
Then by Theorem \ref{4.9}, $\mathbb{K}[x,y,z]/I$ is of dimension $2n^2$
with a $\mathbb K$-basis $\{\ol{x}^i\ol{y}^j, \ol{x}^i\ol{y}^j\ol{z}|i,j\in\mathbb{Z}_n\}$.
Since $\ol{x}^n=1$ and $\ol{z}^2=n\sum_{i=0}^{n-1}\ol{x}^i\ol{z}$,
one gets $(1-\ol{x})\ol{z}^2=0$, and so $((1-\ol{x})\ol{z})^2=0$.
Consequently, the ideal $((1-\ol{x})\ol{z})$ of  $\mathbb{K}[x,y,z]/I$
generated by $(1-\ol{x})\ol{z}$ is contained in $J( \mathbb{K}[x,y,z]/I)$.
Moreover, dim$((\mathbb{K}[x,y,z]/I)/((1-\ol{x})\ol{z}))=n(n+1)$ and
$$(\mathbb{K}[x,y,z]/I)/((1-\ol{x})\ol{z})\cong\mathbb{K}[x,y,z]/(x^n-1,y^n-1,z^2-n^2z, (1-x)z).$$
Let $\pi: \mathbb{K}[x,y,z]\ra \mathbb{K}[x,y,z]/(x^n-1,y^n-1,z^2-n^2z, (1-x)z)$
be the canonical projection. For any integer $k\>0$, let $f_k=\frac{1}{n}\sum_{i=0}^{n-1}q^{ki}x^i$
and $g_k=\frac{1}{n}\sum_{i=0}^{n-1}q^{ki}y^i$ in $\mathbb{K}[x,y,z]$. Then a straightforward verification shows that
$$\begin{array}{c}
\{\pi(f_kg_l), \pi((f_0-\frac{1}{n^2}z)g_l), \pi(\frac{1}{n^2}zg_l)|1\<k\<n-1, 0\<l\<n-1\}\\
\end{array}$$
is a set of orthogonal idempotents, and so it is a full set of orthogonal primitive idempotents
in $\mathbb{K}[x,y,z]/(x^n-1,y^n-1,z^2-n^2z, (1-x)z)$. Therefore,
$$\mathbb{K}[x,y,z]/(x^n-1,y^n-1,z^2-n^2z, (1-x)z)\cong\mathbb{K}^{n(n+1)}.$$
It follows that $J( \mathbb{K}[x,y,z]/I)\subseteq((1-\ol{x})\ol{z})$,
and so $J( \mathbb{K}[x,y,z]/I)=((1-\ol{x})\ol{z})$.
This shows the following proposition.

\begin{proposition}\label{4.10}
Let $J(R_p(H_n(0,q)))$ be the Jacobson radical of $R_p(H_n(0,q))$.
Then $J(R_p(H_n(0,q)))=((1-[S_{1,1}])[P_{0,0}])$ and
$$\begin{array}{rl}
&R_p(H_n(0,q))/J(R_p(H_n(0,q)))\\
\cong&\mathbb{K}[x,y,z]/(x^n-1,y^n-1,z^2-n^2z, (1-x)z)
\cong{\mathbb K}^{n(n+1)}.\\
\end{array}$$
\end{proposition}

\section{\bf The Projective Class Ring of $H_n(1,q)$}\selabel{5}

In this section, we will study the projective class ring of $H_n(1,q)$.
The finite dimensional indecomposable $H_n(1,q)$-modules are classified in \cite{Ch2, Ch4}.
There are $n^2$ simple modules $V(l,r)$ over $H_n(1,q)$,
where $1\leqslant l\leqslant n$ and $r\in\mathbb{Z}_n$.
The simple modules $V(n,r)$ are both projective and injective.
Let $P(l,r)$ be the projective cover of $V(l,r)$. Then $P(l,r)$
is the injective envelope of $V(l,r)$ as well. Moreover,
$P(n,r)\cong V(n,r)$.

Note that $M\ot N\cong N\ot M$ for any modules $M$ and $N$
since $H_n(1, q)$ is a quasitriangular Hopf algebra.
For any $t\in{\mathbb Z}$, let $c(t):=[\frac{t+1}{2}]$ be the integer part of $\frac{t+1}{2}$.
That is, $c(t)$ is the maximal integer with respect to $c(t)\<\frac{t+1}{2}$.
Then $c(t)+c(t-1)=t$.

{\bf Convention}: If $\oplus_{l\<i\<m}M_i$ is a term in a decomposition of a module,
then it disappears when $l>m$.

\begin{lemma}\label{5.1} Let $1\leqslant l, l'\leqslant n$ and $r, r'\in{\mathbb Z}_n$.\\
$(1)$ $V(1, r)\ot V(l, r')\cong V(l, r+r')$.\\
$(2)$ $V(1, r)\ot P(l, r')\cong P(l, r+r')$.\\
$(3)$ If $l\leqslant l'$ and $l+l'\leqslant n+1$, then
$V(l, r)\ot V(l', r')\cong \oplus_{i=0}^{l-1}V(l+l'-1-2i, r+r'+i)$.\\
$(4)$ If $l\leqslant l'$ and $t=l+l'-(n+1)>0$, then
$$\begin{array}{rcl}
V(l, r)\ot V(l', r')&\cong&(\oplus_{i=c(t)}^tP(l+l'-1-2i, r+r'+i))\\
&&\oplus(\oplus_{t+1\<i\<l-1}V(l+l'-1-2i, r+r'+i)).\\
\end{array}$$
$(5)$ If $l\leqslant l'<n$ and $l+l'\leqslant n$, then
$V(l, r)\ot P(l', r')\cong\oplus_{i=0}^{l-1}P(l+l'-1-2i, r+r'+i)$.\\
$(6)$ If $l\leqslant l'<n$ and $t=l+l'-(n+1)\geqslant 0$, then
$$\begin{array}{rl}
V(l, r)\ot P(l', r')
\cong&(\oplus_{i=c(t)}^t2P(l+l'-1-2i, r+r'+i))\\
&\oplus(\oplus_{i=t+1}^{l-1}P(l+l'-1-2i, r+r'+i)).\\
\end{array}$$
$(7)$ If $l'<l<n$ and $l+l'\leqslant n$, then
$$\begin{array}{rcl}
V(l,r)\ot P(l',r')
&\cong&(\oplus_{i=0}^{l'-1}P(l+l'-1-2i, r+r'+i))\\
&&\oplus(\oplus_{i=c(l+l'-1)}^{l-1}2P(n+l+l'-1-2i, r+r'+i)).\\
\end{array}$$
$(8)$ If $l'<l<n$ and $t=l+l'-(n+1)\geqslant 0$, then
$$\begin{array}{rl}
V(l, r)\ot P(l', r')
\cong&(\oplus_{i=c(t)}^t2P(l+l'-1-2i, r+r'+i))\\
&\oplus(\oplus_{i=t+1}^{l'-1}P(l+l'-1-2i, r+r'+i))\\
&\oplus(\oplus_{i=c(l+l'-1)}^{l-1}2P(n+l+l'-1-2i, r+r'+i)).\\
\end{array}$$
$(9)$ If $l<n$, then
$$\begin{array}{rl}
V(n, r)\ot P(l, r')\cong&(\oplus_{i=c(l-1)}^{l-1}2P(n+l-1-2i, r+r'+i))\\
&\oplus(\oplus_{i=1}^{c(n-l)}2P(l-1+2i, r+r'-i)).\\
\end{array}$$
$(10)$ If $l\<l'<n$ and $l+l'\<n$, then
$$\begin{array}{rl}
P(l, r)\ot P(l', r')
\cong &(\oplus_{i=0}^{l-1}2P(l+l'-1-2i, r+r'+i))\\
&\oplus(\oplus_{i=l'}^{l'+l-1}2P(n+l+l'-1-2i, r+r'+i))\\
&\oplus(\oplus_{c(l'+l-1)\<i\<l'-1}4P(n+l+l'-1-2i, r+r'+i))\\
&\oplus(\oplus_{1\<i\<c(n-l-l')}4P(l+l'-1+2i, r+r'-i)).\\
\end{array}$$
$(11)$ If $l\<l'<n$ and $t=l+l'-(n+1)\geqslant 0$, then
$$\begin{array}{rl}
P(l, r)\ot P(l', r')
\cong&(\oplus_{i=c(t)}^{t}4P(l+l'-1-2i, r+r'+i))\\
&\oplus(\oplus_{i=t+1}^{l-1}2P(l+l'-1-2i, r+r'+i))\\
&\oplus(\oplus_{i=l'}^{n-1}2P(n+l+l'-1-2i, r+r'+i))\\
&\oplus(\oplus_{c(l'+l-1)\<i\<l'-1}4P(n+l+l'-1-2i, r+r'+i)).\\
\end{array}$$
\end{lemma}

\begin{proof} It follows from \cite{Ch2, ChenHassenSun}.
\end{proof}

By Lemma \ref{5.1} or \cite[Corollary 3.2]{ChenHassenSun},
the category consisting of semisimple modules and projective modules in
$H_n(1,q)$-mod is a monoidal subcategory of $H_n(1,q)$-mod.
Therefore, we have the following corollary.

\begin{corollary}\label{5.2}
$r_p(H_n(1,q))$ is a free $\mathbb Z$-module with a $\mathbb Z$-basis
$\{[V(k,r)], [P(l,r)]|1\leqslant k\leqslant n, 1\leqslant l\leqslant n-1, r\in\mathbb{Z}_n\}$.
\end{corollary}

\begin{lemma}\label{5.3}
Let $2\leqslant m\leqslant n-1$. Then
$$\begin{array}{c}
V(2,0)^{\ot m}\cong\oplus_{i=0}^{[\frac{m}{2}]}
\frac{m-2i+1}{m-i+1}\binom{m}{i}V(m+1-2i, i).\\
\end{array}$$
\end{lemma}

\begin{proof}
By Lemma \ref{5.1}(3), one can easily check that the isomorphism in the lemma
holds for $m=2$ and $m=3$.
Now let $3<m\leqslant n-1$ and assume
$$\begin{array}{c}
V(2,0)^{\ot(m-1)}\cong \oplus_{i=0}^{[\frac{m-1}{2}]}
\frac{m-2i}{m-i}\binom{m-1}{i}V(m-2i, i).\\
\end{array}$$
If $m=2l$ is even, then by the induction hypothesis and Lemma \ref{5.1}(3), we have
$$\begin{array}{rcl}
V(2,0)^{\ot m}&=&V(2, 0)\ot V(2,0)^{\ot(m-1)}\\
&\cong&\oplus_{i=0}^{l-1}
\frac{2l-2i}{2l-i}\binom{2l-1}{i}V(2, 0)\ot V(2l-2i, i)\\
&\cong&\oplus_{i=0}^{l-1}\frac{2l-2i}{2l-i}\binom{2l-1}{i}(V(2l+1-2i, i)\oplus V(2l-1-2i, i+1))\\
&\cong& V(2l+1,0)\oplus\frac{2}{l+1}\binom{2l-1}{l-1}V(1, l)\\
&&\oplus(\oplus_{i=1}^{l-1}(\frac{2l-2i}{2l-i}
\binom{2l-1}{i}+\frac{2l-2i+2}{2l-i+1}\binom{2l-1}{i-1})V(2l+1-2i, i))\\
&\cong& V(2l+1,0)\oplus\frac{2}{l+1}\binom{2l-1}{l-1}V(1, l)\\
&&\oplus(\oplus_{i=1}^{l-1}\frac{2l+1-2i}{2l+1-i}
\binom{2l}{i}V(2l+1-2i, i))\\
&\cong&\oplus_{i=0}^l\frac{2l+1-2i}{2l+1-i}
\binom{2l}{i}V(2l+1-2i, i))\\
&\cong&\oplus_{i=0}^{[\frac{m}{2}]}\frac{m+1-2i}{m+1-i}
\binom{m}{i}V(m+1-2i, i).\\
\end{array}$$

If $m=2l+1$ is odd, then by the same reason as above, we have
$$\begin{array}{rl}
&V(2,0)^{\ot m}\\
=&V(2, 0)\ot V(2,0)^{\ot(m-1)}\\
\cong&\oplus_{i=0}^l\frac{2l+1-2i}{2l+1-i}
\binom{2l}{i}V(2, 0)\ot V(2l+1-2i, i)\\
\cong&(\oplus_{i=0}^{l-1}\frac{2l+1-2i}{2l+1-i}
\binom{2l}{i}V(2l+2-2i, i)\oplus V(2l-2i, i+1))
\oplus\frac{1}{l+1}\binom{2l}{l}V(2, l)\\
\cong&(\oplus_{i=0}^{l}\frac{2l+1-2i}{2l+1-i}
\binom{2l}{i}V(2l+2-2i, i))
\oplus(\oplus_{i=0}^{l-1}\frac{2l+1-2i}{2l+1-i}
\binom{2l}{i}V(2l-2i, i+1))\\
\cong&(\oplus_{i=0}^{l}\frac{2l+1-2i}{2l+1-i}
\binom{2l}{i}V(2l+2-2i, i))
\oplus(\oplus_{i=1}^{l}\frac{2l+3-2i}{2l+2-i}
\binom{2l}{i-1}V(2l+2-2i, i))\\
\cong& V(2l+2,0)
\oplus(\oplus_{i=1}^l(\frac{2l+1-2i}{2l+1-i}\binom{2l}{i}
+\frac{2l+3-2i}{2l+2-i}\binom{2l}{i-1})V(2l+2-2i, i))\\
\cong& V(2l+2,0)
\oplus(\oplus_{i=1}^l\frac{2l+2-2i}{2l+2-i}\binom{2l+1}{i}V(2l+2-2i, i))\\
\cong&\oplus_{i=0}^l\frac{2l+2-2i}{2l+2-i}\binom{2l+1}{i}V(2l+2-2i, i)\\
\cong& \oplus_{i=0}^{[\frac{m}{2}]}\frac{m+1-2i}{m+1-i}\binom{m}{i}V(m+1-2i, i).\\
\end{array}$$
\end{proof}

Throughout the following, let $x=[V(1,1)]$ and $y=[V(2,0)]$ in
$r_p(H_n(1,q))$.

\begin{corollary}\label{5.4}
The following equations hold in $r_p(H_n(1,q))$ (or $r(H_n(1,q))$):\\
$(1)$ $x^n=1$ and $[V(m,i)]=x^i[V(m,0)]$
for all $1\<m\<n$ and $i\in\mathbb Z$;\\
$(2)$ $[P(m,i)]=x^i[P(m,0)]$ for all $1\<m<n$ and $i\in\mathbb Z$;\\
$(3)$ $y[V(n,0)]=x[P(n-1,0)]$;\\
$(4)$ $y[P(1,0)]=[P(2,0)]+2x[V(n,0)]$;\\
$(5)$ $y[P(n-1,0)]=2[V(n,0)]+x[P(n-2,0)]$;\\
$(6)$ $y[P(m,0)]=[P(m+1,0)]+x[P(m-1,0)]$ for all $2\<m\<n-2$;\\
$(7)$ $[V(m+1,0)]=y^{m}-\sum_{i=1}^{[\frac{m}{2}]}\frac{m+1-2i}{m+1-i}\binom{m}{i}
x^i[V(m+1-2i,0)]$ for all $2\<m<n$.
\end{corollary}

\begin{proof}
It follows from Lemmas \ref{5.1} and \ref{5.3}.
\end{proof}

\begin{proposition}\label{5.5}
The commutative ring $r_p(H_n(1,q))$ is generated by $x$ and $y$.
\end{proposition}

\begin{proof}
Let $R$ be the subring of $r(H_n(1,q))$ generated by $x$ and $y$.
Then $R\subseteq r_p(H_n(1,q))$. By Corollary \ref{5.4}(1), one gets that
$[V(1,i)]=x^i\in R$ and $[V(2,i)]=x^iy\in R$ for all $i\in\mathbb{Z}_n$.
Now let $2\<m<n$ and assume $[V(l,i)]\in R$ for all $1\<l\<m$ and $i\in\mathbb{Z}_n$.
Then by Corollary \ref{5.4}(1) and (7), one gets that
$[V(m+1,i)]=x^i[V(m+1,0)]=x^iy^{m}-\sum_{j=1}^{[\frac{m}{2}]}\frac{m+1-2j}{m+1-j}\binom{m}{j}
x^{i+j}[V(m+1-2j,0)]\in R$
for all $i\in\mathbb{Z}_n$. Thus, we have proven that
$[V(m,i)]\in R$ for all $1\<m\<n$ and $i\in\mathbb{Z}_n$.
In particular, $[V(n,i)]\in R$ for all $i\in\mathbb{Z}_n$.

By Corollary \ref{5.4}(2) and (3), $[P(n-1, i)]=x^i[P(n-1,0)]=x^{i-1}y[V(n,0)]\in R$ for all $i\in\mathbb{Z}_n$.
Then by Corollary \ref{5.4}(2) and (5), $[P(n-2, i)]=x^i[P(n-2,0)]=x^{i-1}(y[P(n-1, 0)]-2[V(n, 0)])\in R$
for any $i\in\mathbb{Z}_n$. Now let $1<m\<n-2$ and assume that
$[P(l,i)]\in R$ for all $m\<l<n$ and $i\in\mathbb{Z}_n$.
Then by Corollary \ref{5.4}(2) and (6), we have
$[P(m-1, i)]=x^i[P(m-1,0)]=x^{i-1}(y[P(m, 0)]-[P(m+1, 0)])\in R$.
Thus, we have shown that $[P(m,i)]\in R$ for all $1\<m<n$ and
$i\in\mathbb{Z}_n$. Then it follows from Corollary \ref{5.2} that
$R=r_p(H_n(1,q))$. This completes the proof.
\end{proof}

\begin{lemma}\label{5.6}
$(1)$ $[V(m,0)]=\sum_{i=0}^{[\frac{m-1}{2}]}(-1)^i\binom{m-1-i}{i}x^iy^{m-1-2i}$ for all $1\<m\<n$.\\
$(2)$ Let $1\<m\<n-1$. Then
$$\begin{array}{c}
[P(m,0)]=(\sum_{i=0}^{[\frac{n-m}{2}]}(-1)^i\frac{n-m}{n-m-i}\binom{n-m-i}{i}x^{m+i}y^{n-m-2i})[V(n,0)].\\
\end{array}$$
\end{lemma}

\begin{proof}
(1) It is similar to \cite[Lemma 3.2]{ZWLC}.\\
(2) Note that $\frac{n-m}{n-m-i}\binom{n-m-i}{i}$ is a positive integer for any
$1\<m\<n-1$ and $0\<i\<[\frac{n-m}{2}]$. We prove the equality by induction on $n-m$. If $m=n-1$, then
by Corollary \ref{5.4}(1) and (3), $[P(n-1,0)]=x^{-1}y[V(n,0)]=x^{n-1}y[V(n,0)]$,
as desired. If $m=n-2$, then by Corollary \ref{5.4}(1) and (5), we have
$[P(n-2,0)]=x^{-1}y[P(n-1,0)]-2x^{-1}[V(n,0)]=(x^{n-2}y^2-2x^{n-1})[V(n,0)]$,
as desired. Now let $1\<m<n-2$. Then by Corollary \ref{5.4}(1) and (6), and the
induction hypotheses, we have
$$\begin{array}{rl}
[P(m,0)]=&x^{-1}y[P(m+1,0)]-x^{-1}[P(m+2,0)]\\
=&x^{-1}y(\sum_{i=0}^{[\frac{n-m-1}{2}]}(-1)^i\frac{n-m-1}{n-m-1-i}\binom{n-m-1-i}{i}x^{m+1+i}y^{n-m-1-2i})[V(n,0)]\\
&-x^{-1}(\sum_{i=0}^{[\frac{n-m-2}{2}]}(-1)^i\frac{n-m-2}{n-m-2-i}\binom{n-m-2-i}{i}x^{m+2+i}y^{n-m-2-2i})[V(n,0)]\\
=&(\sum_{i=0}^{[\frac{n-m-1}{2}]}(-1)^i\frac{n-m-1}{n-m-1-i}\binom{n-m-1-i}{i}x^{m+i}y^{n-m-2i})[V(n,0)]\\
&+(\sum_{i=1}^{[\frac{n-m}{2}]}(-1)^{i}\frac{n-m-2}{n-m-1-i}\binom{n-m-1-i}{i-1}x^{m+i}y^{n-m-2i})[V(n,0)].\\
\end{array}$$
If $n-m$ is odd, then $[\frac{n-m-1}{2}]=\frac{n-m-1}{2}=[\frac{n-m}{2}]$, and hence
$$\begin{array}{rl}
&\sum_{i=0}^{[\frac{n-m-1}{2}]}(-1)^i\frac{n-m-1}{n-m-1-i}\binom{n-m-1-i}{i}x^{m+i}y^{n-m-2i}\\
&+\sum_{i=1}^{[\frac{n-m}{2}]}(-1)^{i}\frac{n-m-2}{n-m-1-i}\binom{n-m-1-i}{i-1}x^{m+i}y^{n-m-2i}\\
=&x^my^{n-m}+\sum_{i=1}^{[\frac{n-m}{2}]}(-1)^i(\frac{n-m-1}{n-m-1-i}\binom{n-m-1-i}{i}\\
&+\frac{n-m-2}{n-m-1-i}\binom{n-m-1-i}{i-1})x^{m+i}y^{n-m-2i}\\
=&\sum_{i=0}^{[\frac{n-m}{2}]}(-1)^i\frac{n-m}{n-m-i}\binom{n-m-i}{i}x^{m+i}y^{n-m-2i}.\\
\end{array}$$
If $n-m$ is even, then $[\frac{n-m-1}{2}]=\frac{n-m-2}{2}=[\frac{n-m}{2}]-1$, and hence
$$\begin{array}{rl}
&\sum_{i=0}^{[\frac{n-m-1}{2}]}(-1)^i\frac{n-m-1}{n-m-1-i}\binom{n-m-1-i}{i}x^{m+i}y^{n-m-2i}\\
&+\sum_{i=1}^{[\frac{n-m}{2}]}(-1)^{i}\frac{n-m-2}{n-m-1-i}\binom{n-m-1-i}{i-1}x^{m+i}y^{n-m-2i}\\
=&x^my^{n-m}+\sum_{i=1}^{[\frac{n-m}{2}]-1}(-1)^i(\frac{n-m-1}{n-m-1-i}\binom{n-m-1-i}{i}\\
&+\frac{n-m-2}{n-m-1-i}\binom{n-m-1-i}{i-1})x^{m+i}y^{n-m-2i}+(-1)^{\frac{n-m}{2}}2x^{\frac{n+m}{2}}\\
=&\sum_{i=0}^{[\frac{n-m}{2}]}(-1)^i\frac{n-m}{n-m-i}\binom{n-m-i}{i}x^{m+i}y^{n-m-2i}.\\
\end{array}$$
Therefore, $[P(m,0)]=(\sum_{i=0}^{[\frac{n-m}{2}]}(-1)^i\frac{n-m}{n-m-i}\binom{n-m-i}{i}x^{m+i}y^{n-m-2i})[V(n,0)]$.
\end{proof}

\begin{proposition}\label{5.7}
In $r_p(H_n(1,q))$ (or $r(H_n(1,q))$), we have
$$\begin{array}{c}
(\sum_{i=0}^{[\frac{n}{2}]}(-1)^i\frac{n}{n-i}\binom{n-i}{i}x^iy^{n-2i}-2)
(\sum_{i=0}^{[\frac{n-1}{2}]}(-1)^i\binom{n-1-i}{i}x^iy^{n-1-2i})=0.\\
\end{array}$$
\end{proposition}

\begin{proof}
By Lemma \ref{5.6}(2), we have
$$\begin{array}{c}
x^{-1}y[P(1,0)]=(\sum_{i=0}^{[\frac{n-1}{2}]}(-1)^i\frac{n-1}{n-1-i}\binom{n-1-i}{i}x^{i}y^{n-2i})[V(n,0)].\\
\end{array}$$
On the other hand, by Corollary \ref{5.4}(4) and Lemma \ref{5.6}(2), we have
$$\begin{array}{rl}
x^{-1}y[P(1,0)]=&x^{-1}[P(2,0)]+2[V(n,0)]\\
=&(\sum_{i=0}^{[\frac{n-2}{2}]}(-1)^i\frac{n-2}{n-2-i}\binom{n-2-i}{i}x^{i+1}y^{n-2-2i}+2)[V(n,0)]\\
=&(\sum_{i=1}^{[\frac{n}{2}]}(-1)^{i-1}\frac{n-2}{n-1-i}\binom{n-1-i}{i-1}x^{i}y^{n-2i}+2)[V(n,0)].\\
\end{array}$$
Therefore, one gets
$$\begin{array}{rl}
&(\sum_{i=0}^{[\frac{n-1}{2}]}(-1)^i\frac{n-1}{n-1-i}\binom{n-1-i}{i}x^{i}y^{n-2i})[V(n,0)]\\
=&(\sum_{i=1}^{[\frac{n}{2}]}(-1)^{i-1}\frac{n-2}{n-1-i}\binom{n-1-i}{i-1}x^{i}y^{n-2i}+2)[V(n,0)],\\
\end{array}$$
which is equivalent to
$$\begin{array}{rl}
&(\sum_{i=0}^{[\frac{n-1}{2}]}(-1)^i\frac{n-1}{n-1-i}\binom{n-1-i}{i}x^{i}y^{n-2i}\\
&-\sum_{i=1}^{[\frac{n}{2}]}(-1)^{i-1}\frac{n-2}{n-1-i}\binom{n-1-i}{i-1}x^{i}y^{n-2i}-2)[V(n,0)]=0.\\
\end{array}$$
Then a computation similar to the proof of Lemma \ref{5.6} shows that
$$\begin{array}{rl}
&\sum_{i=0}^{[\frac{n-1}{2}]}(-1)^i\frac{n-1}{n-1-i}\binom{n-1-i}{i}x^{i}y^{n-2i}
-\sum_{i=1}^{[\frac{n}{2}]}(-1)^{i-1}\frac{n-2}{n-1-i}\binom{n-1-i}{i-1}x^{i}y^{n-2i}-2\\
=&\sum_{i=0}^{[\frac{n}{2}]}(-1)^i\frac{n}{n-i}\binom{n-i}{i}x^{i}y^{n-2i}-2.\\
\end{array}$$
Thus, the proposition follows from Lemma \ref{5.6}(1).
\end{proof}

\begin{corollary}\label{5.8}
$\{x^ly^m|0\<l\<n-1, 0\<m\<2n-2\}$ is a $\mathbb Z$-basis of $r_p(H_n(1,q))$.
\end{corollary}

\begin{proof}
By Corollary \ref{5.4}(1), $x^n=1$. By Proposition \ref{5.7}, we have
$$\begin{array}{rl}
y^{2n-1}=&-\sum_{i=1}^{[\frac{n-1}{2}]}(-1)^i\binom{n-1-i}{i}x^iy^{2n-1-2i}\\
&-\sum_{i=1}^{[\frac{n}{2}]}(-1)^i\frac{n}{n-i}\binom{n-i}{i}x^iy^{2n-1-2i}+2y^{n-1}\\
&-(\sum_{i=1}^{[\frac{n}{2}]}(-1)^i\frac{n}{n-i}\binom{n-i}{i}x^iy^{n-2i}-2)
(\sum_{i=1}^{[\frac{n-1}{2}]}(-1)^i\binom{n-1-i}{i}x^iy^{n-1-2i}).\\
\end{array}$$
Then it follows from Proposition \ref{5.5} that $r_p(H_n(1,q))$ is generated,
as a $\mathbb Z$-module, by $\{x^ly^m|0\<l\<n-1, 0\<m\<2n-2\}$.
By Corollary \ref{5.2}, $r_p(H_n(1,q))$ is a free $\mathbb Z$-module of rank $n(2n-1)$,
and hence $\{x^ly^m|0\<l\<n-1, 0\<m\<2n-2\}$ is a $\mathbb Z$-basis of $r_p(H_n(1,q))$.
\end{proof}

\begin{theorem}\label{5.9}
Let $\mathbb{Z}[x,y]$ be the polynomial ring in two variables $x$ and $y$,
 and $I$ the ideal of $\mathbb{Z}[x,y]$ generated by $x^n-1$ and
 $$\begin{array}{c}
 (\sum_{i=0}^{[\frac{n}{2}]}(-1)^i\frac{n}{n-i}\binom{n-i}{i}x^iy^{n-2i}-2)
(\sum_{i=0}^{[\frac{n-1}{2}]}(-1)^i\binom{n-1-i}{i}x^iy^{n-1-2i}).\\
\end{array}$$
Then $r_p(H_n(1,q))$ is isomorphic to the quotient ring $\mathbb{Z}[x,y]/I$.
\end{theorem}

\begin{proof}
By Proposition \ref{5.5}, there is a ring epimorphism
$\phi: \mathbb{Z}[x,y]\ra r_p(H_n(1,q))$ given by
$\phi(x)=[V(1,1)]$ and $\phi(y)=[V(2,0)]$.
By Corollary \ref{5.4}(1) and Proposition \ref{5.7}, $\phi(I)=0$.
Hence $\phi$ induces a ring epimorphism
$\ol{\phi}: \mathbb{Z}[x,y]/I\ra r_p(H_n(1,q))$ such that $\phi=\ol{\phi}\circ\pi$,
where $\pi: \mathbb{Z}[x,y]\ra\mathbb{Z}[x,y]/I$ is the canonical projection.
Let $\ol{u}=\pi(u)$ for any $u\in\mathbb{Z}[x,y]$.
Then by the definition of $I$ and the proof of Corollary \ref{5.8}, one knows that
$\mathbb{Z}[x,y]/I$ is generated, as a $\mathbb Z$-module,
by $\{\ol{x}^l\ol{y}^m|0\<l\<n-1, 0\<m\<2n-2\}$.
For any $0\<l\<n-1$ and $0\<m\<2n-2$, we have
$\ol{\phi}(\ol{x}^l\ol{y}^m)=\ol{\phi}(\ol{x})^l\ol{\phi}(\ol{y})^m=\phi(x)^l\phi(y)^m
=[V(1,1)]^l[V(2,0)]^m$. By Corollary \ref{5.8},
$\{[V(1,1)]^l[V(2,0)]^m|0\<l\<n-1, 0\<m\<2n-2\}$ is a linearly independent set over $\mathbb Z$,
which implies that $\{\ol{x}^l\ol{y}^m|0\<l\<n-1, 0\<m\<2n-2\}$ is also a linearly independent set
over $\mathbb Z$. It follows that $\{\ol{x}^l\ol{y}^m|0\<l\<n-1, 0\<m\<2n-2\}$
is a $\mathbb Z$-basis of $\mathbb{Z}[x,y]/I$. Consequently, $\ol{\phi}$ is a $\mathbb Z$-module
isomorphism, and so it is a ring isomorphism.
\end{proof}

\rm
\begin{center}
\Large {\bf Acknowledgment}
\end{center}

This work is supported by NSF of China (No. 11571298)
and TAPP of Jiangsu Higher Education Institutions (No. PPZY2015B109).

\end{document}